\documentclass{amsart}
\usepackage{amsthm,siunitx}
\usepackage{amssymb,amsmath,amscd, eurosym,mathtools}

\newtheorem{theorem}{Theorem}[section]

\newtheorem{conjecture}[theorem]{Conjecture}

\theoremstyle{definition}
  
\theoremstyle{remark}

\newtheorem{acknowledgement}{Acknowledgement}

\begin{document}

\large
\title[A short proof of Erd\"{o}s--Straus conjecture for every $n\equiv 13 \textrm{ mod }24$
]{A short proof  of Erd\"{o}s--Straus conjecture for every $n\equiv 13 \textrm{ mod }24$
 }
\thanks{Last updated: March 10, 2021}

\author[M. Gionfriddo]{Mario Gionfriddo}
	\address{Dipartimento di Matematica e Informatica\\
		Viale A. Doria, 6 - 95100 - Catania, Italy}
	\email{gionfriddo@dmi.unict.it} \urladdr{https://www.researchgate.net/profile/Mario\_Gionfriddo}

\author[E. Guardo]{Elena Guardo}
\address{Dipartimento di Matematica e Informatica\\
		Viale A. Doria, 6 - 95100 - Catania, Italy}
	\email{guardo@dmi.unict.it} \urladdr{www.dmi.unict.it/guardo}


\begin{abstract} The Erd\"{o}s--Straus conjecture states that the equation  $\frac{4}{n}=\frac{1}{x}+\frac{1}{y}+\frac{1}{z}$ has positive integer solutions $x,y,z$ for every postive integers $n\geq 2$. In this short note we find explicity the solutions of the famous conjecture for the case $n\equiv13 \textrm{ mod } 24.$ 
\end{abstract}

\maketitle

\section{Introduction}
In Number Theory there are many unsolvable  problems that still attract a lot of attention. Among all, there is a famous conjecture of Erd\"{o}s--Straus which states that, for all positive integers $n\geq 2$, the rational number $\frac{4}{n}$ can be expressed as the sum of three positive unit fractions. Specifically, in 1948  P. Erd\"{o}s and E. G. Straus formulated the following:  

\begin{conjecture}\label{ES}  For every positive integer $n\geq 2$ there exist positive integers $x,y,z$ such that:
\begin{equation}\label{cong}
\frac{4}{n}=\frac{1}{x}+\frac{1}{y}+\frac{1}{z}.
\end{equation}
\end{conjecture}

This conjecture has  attracted a lot of attention not only among reseachers in Number Theory but also among many people involved among the different areas in Mathematics, such as L. Bernstein \cite{Be},  M. B. Crawford \cite{Crawford 2019},  M. Di Giovanni,  S. Gallipoli, M. Gionfriddo \cite{DGG}, C. Elsholtz and T. Tao \cite{ET}, J. Guanouchi \cite{Ga1,Ga2}, S. K. Jeswal and S. Chakraverty \cite{JC}, L. J. Mordell \cite{Mo}, D. J. Negash \cite{Ne},  R. Obl\'ath \cite{Ob}, L. A. Rosati \cite{R}, J. W. Sander \cite{Sa}, R. C. Vaughan \cite{V}, K. Yamamoto \cite{Y}, just to cite some of them. For example, Swett in \cite{Sw} has established validity of the conjecture for all $n\leq 10^{14}$ and it also appears in The Penguin Dictionary of Curious and Interesting Numbers, \cite{Pin}, or  S. E. Salez \cite{S} that proved that the conjecture holds for all $n\leq 10^{17}$.

It is not clear if Conjecture \ref{ES} is true or not since there are many papers in which some authors adfirm to have proved that the conjecture is true and others in which the authors proved that it is false (see \cite{Z}). Since we have found some mistakes in some of these published papers, it is our opinion that this conjecture is still open.

L. J. Mordell, in  \cite{Mo}, Chapter 30, Section 1,  has only proven that the conjecture is true for all $n$ except possibly cases in which $n$ is congruent to $1,121,169,289,361,529$ mod $840$, but he did not write explicity the solution as we do in Theorem \ref{mainresult}.
In all the known literature,  for the case $n\not\equiv 1\textrm{ mod } 24$  the authors prove only the existence of the solutions without writing them explicity.

Since it is known that for $n\not\equiv 1\textrm{ mod } 24$ solutions of  (\ref{cong}) can be found constructively and that for the remaining residue class $n\equiv 1\textrm{ mod } 24$ the problem can be reduced to smaller sets with some exceptions (see \cite{Mo} and \cite{Sa}), we have divided the problem into two parts, $n\equiv 1 \textrm{ mod } 24$  and  $n\equiv 13 \textrm{ mod } 24$.  In Section \ref{known}, we recall all the known cases in the literature  and in Section \ref{short} we propose our short proof of the conjecture for the case $n\equiv 13  \textrm{ mod } 24$. Therefore, to completely solve Conjecture \ref{ES}, it remains to find solutions for every $n\equiv 1  \textrm{ mod } 24$.

\begin{acknowledgement}
The authors were supported by  Universit\`{a} degli Studi di Catania, \lq\lq Piano della Ricerca PIACERI 2018/2020 Linea di intervento 2\rq\rq. 
The second author is a member of GNSAGA of INdAM (Italy).
\end{acknowledgement}

\section{Known cases} \label{known}
\subsection{ $n$  even}\label{neven}

The case $n$ even is easily solvable, as we can see in the following theorem.

\begin{theorem}\label{case0mod4} Conjecture \ref{ES}  is true  for $n$ even.
\end{theorem}
\begin{proof}
Let $n\equiv 0,2 \textrm{ mod } 4$.
\begin{enumerate}
\item If  $n=4k$, then for: $ x=2k, y=4k, z=4k$, it is possible to verify that the conjecture is true.

\item If 4$n=4k+2$, then for: $ x=2k+1, y=4k+2, z=4k+2$, it is possible to verity that the conjecture is true.
\end{enumerate}
\end{proof}

\subsection{$n$  odd}\label{nodd} 

The following results are already known for the case $n$ odd.

\begin{theorem}\label{case3mod4}

Conjecture \ref{ES} is true for $n\equiv 3\textrm{ mod }4$.
\end{theorem}
\begin{proof}
If  $n=4k+3$, then for: $x=2k+2, y=2k+2, z=(k+1)(4k+3)$, it is possible to verity that the conjecture is true (see \cite{CDGG}).	
\end{proof}

\begin{theorem}\label{case3mod6}
Conjecture \ref{ES} is true  for $n\equiv 3,5 \textrm{ mod } 6$.
\end {theorem}
\begin{proof}  Let $n\equiv 3,5 \textrm{ mod } 6$ (see \cite{CDGG}).

If  $n=6k+3$, then for: $x=6k+3, y=2k+2, z=(2k+1)(2k+2)$,  it is possible to verify that the conjecture is true.
If $ n=6k+5$, then for:  $x=6k+5, y=2k+2, z=(6k+5)(2k+2)$, it is possible to verify that the conjecture is true. 	
\end{proof}

Collecting together the results proved in Theorems \ref{case3mod4} and  \ref{case3mod6}, we have that: 
 
\begin{theorem}\label{case3mod13}\cite{CDGG} 
Conjecture \ref{ES} is true  for every $ n\equiv 3,5,7,9,11   \textrm{ mod } 12$.
\end{theorem}

\begin{proof}  Indeed:
\begin{enumerate}
\item if  $n\equiv 3\textrm{ mod }12$, then it is $n\equiv 3\textrm{ mod }$ and also $n\equiv 3 \textrm{ mod } 6$;
\item if  $n\equiv 5 \textrm{ mod } 12$, then  it is $n\equiv 1\textrm{ mod } 4$, but also $n\equiv 5  \textrm{ mod } 6$;
\item if  $n\equiv 7\textrm{ mod } 12$,  then it is $n\equiv1\textrm{ mod }6$, but also $n\equiv 3  \textrm{ mod } 4$;
\item if $n\equiv 9 \textrm{ mod }12$, then  it is $n\equiv 1 \textrm{ mod } 4$, but also $n\equiv 3  \textrm{ mod } 6$;
\item if $n\equiv 11\textrm{ mod } 12$, then it is $n\equiv 3\textrm{ mod } 4$ and also $n\equiv 5 \textrm{ mod } 6$.
\end{enumerate}
\end{proof}

Therefore, it remains to examine the case $n\equiv 1 \textrm{ mod } 12$.  We have divided the problem into two parts, $n\equiv 1 \textrm{ mod } 24$  and  $n\equiv 13 \textrm{ mod } 24$, and get a very short proof of the conjecture for the case $n\equiv 13  \textrm{ mod } 24$.

\section{Main Result: Solutions for $n\equiv 13\textrm{ mod } 24$.}\label{short} 


\begin{theorem}\label{mainresult}  Conjecture \ref{ES} is true for every $n\equiv13  \textrm{ mod } 24$.
\end{theorem}
\begin{proof}:  Let $n=24k+13$, for any non negative integer $k$. If:
\begin{eqnarray*}
x&=& (2(k+1)(24k+13)),\\
y&=& (2(3k+2)),\\
z&=&(2(k+1)(24k+13)(3k+2)),
\end{eqnarray*}
\noindent 
then it follows:
{ \Large{
\begin{eqnarray*}
&\frac{1}{2(k+1)(24k+13)} +\frac{1}{2(3k+2)} +\frac{1}{2(k+1)(24k+13)(3k+2)} &=\\
&=\frac{(3k+2)+(k+1)(24k+13)+1}{2(k+1)(24k+13)(3k+2)} &=\\
&=\frac{8(3k^2+5k+2)}{2(k+1)(24k+13)(3k+2)}&=\\
&=\frac{8(3k+2)(k+1)}{2(k+1)(24k+13)(3k+2)}&=\\
&=\frac{4}{(24k+13)}
\end{eqnarray*}}}

\noindent which proves the statement.
\end{proof}

\end{document}